\documentclass{amsart}
\usepackage[foot]{amsaddr}
\author{Sandro Gallo and Pablo M. Rodriguez} 
\title{Frog models on trees through renewal theory} 
\email{sandro.gallo@ufscar.br}
\email{pablor@icmc.usp.br}
\usepackage{amsfonts,amssymb,latexsym,xspace,epsfig,graphics,color}
\usepackage{amsmath,enumerate,stmaryrd,xy}
\usepackage{bbm}
\usepackage{subfigure}
\usepackage{dsfont}
\usepackage{a4wide}

\usepackage{tikz}
\usetikzlibrary{arrows,snakes,backgrounds,trees}

\newtheorem{thm}{Theorem}[section]
\newtheorem{obs}{Observation}[section]
\newtheorem{prop}{Proposition}[section]
\newtheorem{coro}{Corollary}[section]
\newtheorem{lem}{Lemma}[section]

\newcommand{\bN}{\mathbb{N}}

\begin{document}

%
%
\maketitle
\begin{abstract}
This paper studies a class of growing systems of random walks on {regular} trees, known as  \emph{frog models with geometric lifetime} in the  literature. With the help of results from renewal theory, we derive new bounds for their critical parameters. Our approach also improve the bounds of the literature for the critical parameter of a percolation model on trees called \emph{cone percolation}.
\end{abstract}

   \section{Introduction}

In this work we explore a new approach for studying the localization of the critical parameter of a growing system of random walks on {regular} trees, known as  \emph{frog models with lifetime} in the  literature. This approach is based on a link between undelayed renewal sequences and a frog model on directed {regular} trees.  

On  {general} infinite connected graphs, the original frog model with geometric lifetime is inspired in a process of information transmission on a moving population and may be informally described as follows. Assume that at time $0$ each vertex of the graph has one particle which may be in one of two states: active or inactive. Each active particle performs, independently of the others, a discrete time random walk through the vertices of the {graph} during a random number of steps, geometrically distributed with parameter $1-p$, for some $p\in(0,1)$. This time is called the lifetime of the active particle, and once it is reached, we assume that the particle is removed from the system. Removed particles play no role in the spreading procedure. On the other hand, if an active particle jumps to a vertex containing an inactive one, then the inactive particle becomes active and starts  an independent random walk on the {graph}. Usually, it is assumed that the process starts with one active particle at {a fixed vertex, and inactive particles everywhere else}. We refer the reader to \cite{alves2002} for a formal definition of the model. 

One of the main questions of interest is the survival, or not, of a particular realization of the process, that is, whether or not there is, at any time, at least one active particle on the {graph}. This question has been addressed on homogeneous trees and other infinite graphs like hypercubic lattices by \cite{alves2002} and \cite{lebensztayn/machado/popov/2005}. 
On infinite trees, a simple coupling argument shows that the probability of survival of the frog model on $\mathbb{T}_d$, the homogeneous tree of degree $d+1$ {(i.e., any vertex has $d+1$ neighbors)}, which we denote by $\theta(d,p)$, is a nondecreasing function of $p$, and therefore we may define the critical parameter as $p_c(d ):=\inf\{p:\theta(d,p)>0\}.$   
In \cite{alves2002} the authors present a necessary and sufficient condition for the existence of phase transition on $\mathbb{T}_d$, i.e. for $0<p_c(d)<1$, and also give {the bounds $(d+1)/(2d+1)\le p_c(d)\le(d+1)/(2d-2)$} for the critical parameter. Their upper bound was later improved {in \cite{lebensztayn/machado/popov/2005}, where it is proved that $p_c(d)\le(d+1)/(2d)$}. Similar results are obtained by \cite{lebensztayn/machado/martinez/2006}, for the modified version of the model in which each active particle performs a self-avoiding discrete time random walk on the tree. The present work is in the line of these  papers. { We refer the reader to \cite{alves2002,hoffman2017} and references therein for results related to the behavior of the frog model when the lifetimes of the particles are not restricted. More precisely, \cite{alves2002} obtains a shape theorem for the model on the $d$-dimensional hypercubic lattice; while \cite{hoffman2017} studies recurrence/transience for the model on trees. The latter is indeed a topic of current research.}

The paper is organized as follows. We define in Section \ref{ss:frogs_trees} a frog model on directed trees, and obtain tight bounds for the critical parameter as function of $d$. In Section \ref{sec:applications} we will compare our model to three models (the original frog model, the self-avoiding frog model and the cone percolation with geometric radius), and improve several previously known bounds. Section \ref{sec:renewal} is an interlude on renewal processes, containing results that we will use for our proofs. Section \ref{sec:proofs} contains the proofs, our  strategy will be to use the link, originally showed in \cite{gallo/garcia/junior/rodriguez/2014}, between {some} long range information propagation models on $\bN$ and undelayed renewal sequences.

\section{Frog model on directed trees}\label{ss:frogs_trees}
\subsection{Definition of the model} 
{
We consider the directed rooted tree $\overrightarrow{\mathbb{T}}_d=(\mathcal{V},\overrightarrow{\mathcal{E}})$, defined by making all the edges of the $d$-regular tree $\mathbb{T}_d=(\mathcal{V},\mathcal{E})$ point away from the root. We define the distance between $u,v\in\mathcal{V}$, denoted by $d(u,v) $, as the number of edges in the unique path connecting them. We write $u< v$ if $u\neq v$ and $u$ is one of the vertices of the path connecting the root to $v$.} In the present paper, we consider a  frog model on $\overrightarrow{\mathbb{T}}_d$, that is, the active particles try to activate other particles localized away from the root, as shown on Figure \ref{fig:frogs}.

In order to define the model, let $(\Omega,\mathcal{F},\mathbb{P})$ be a probability space where the following random objects are independent and well defined for any $v\in \mathcal{V}$:  $(S_{n}^{v})_{n\geq 0}$ is a  discrete time symmetric random walk on $\overrightarrow{\mathbb{T}}_d$ starting from $v$ and $ \mathcal{T}_v$ a $\bN$-valued random variable satisfying $\mathbb{P}(\mathcal{T}_v\ge n)=c(dq)^n$, for some $c\in(0,1]$ and  {$q\in(0,1)$ such that $c(dq)^n <1$ for any $n{\geq 1}$.  
The random walk $(S_{n}^{v})_{n\geq 0}$ represents the trajectory of the  particle starting at $v$ and $\mathcal{T}_v$ represents its lifetime. We now define the  {truncated random walk starting at} $v$, $(R_n^v)_{n\ge0}$, by
$$R_{n}^{v}:=\left\{
\begin{array}{ll}
S_n^{v}, & n<\mathcal{T}_{v};\\[.2cm]
S_{\mathcal{T}_{v}-1}^{v}, & n\geq \mathcal{T}_{v}.\\[.2cm]
\end{array}\right.
$$
Observe that, by symmetry, for any $n\ge1$ and any $u\in \partial T_{v}^n:=\{u\in \mathcal{V}: u>v, \text{dist}(v,u)=n\}$  
\begin{equation}\label{eq:Rn}
\mathbb{P}(R_n^{v} = u)=c q^n.
\end{equation}

\begin{figure}[h]
\begin{center}
\subfigure[][$t=0$]{

\tikzstyle{level 1}=[sibling angle=120]
\tikzstyle{level 2}=[sibling angle=100]
\tikzstyle{level 3}=[sibling angle=80]
\tikzstyle{every node}=[fill]

\tikzstyle{edge from parent}=[segment length=0.8mm,segment angle=10,draw,->]
\begin{tikzpicture}[scale=0.5,grow cyclic,shape=circle,minimum size = 2pt,inner sep=1.5pt,level distance=13mm,
                    cap=round]
\node {} child [\A] foreach \A in {black,black,black}
    { node {} child [color=\A!50!\B] foreach \B in {black,black}
        { node {} child [color=\A!50!\B!50!\C] foreach \C in {black,black}
            { node {} }
        }
    };
    
\filldraw [gray] (0,0) circle (3pt);
\draw [dashed,->] (0.1,0.2) to [bend left=40] (1.2,0.2);    

\end{tikzpicture}

}
\qquad
\subfigure[][$t=1$]{

\tikzstyle{level 1}=[sibling angle=120]
\tikzstyle{level 2}=[sibling angle=100]
\tikzstyle{level 3}=[sibling angle=80]
\tikzstyle{every node}=[fill]
\tikzstyle{edge from parent}=[segment length=0.8mm,segment angle=10,draw,->]
\begin{tikzpicture}[scale=0.5,grow cyclic,shape=circle,minimum size = 2pt,inner sep=1.5pt,level distance=13mm,
                    cap=round]
\node [white] {} child [\A] foreach \A in {black,black,black}
    { node {} child [color=\A!50!\B] foreach \B in {black,black}
        { node {} child [color=\A!50!\B!50!\C] foreach \C in {black,black}
            { node {} }
        }
    };

\draw  (0,0) circle (2.1pt);        
\filldraw [white] (1.3,0) circle (2pt);        
\filldraw [gray] (1.25,0.15) circle (3pt);    
\filldraw [gray] (1.35,-0.1) circle (3pt);    
\draw [dashed,->] (1.26,0.35) to [bend left=40] (1.95,1);    
\draw [dashed,->] (1.6,-0.1) to [bend right=40] (2.16,0.8);    

\end{tikzpicture}

}
\qquad
\subfigure[][$t=2$]{

\tikzstyle{level 1}=[sibling angle=120]
\tikzstyle{level 2}=[sibling angle=100]
\tikzstyle{level 3}=[sibling angle=80]
\tikzstyle{every node}=[fill]
\tikzstyle{edge from parent}=[segment length=0.8mm,segment angle=10,draw,->]
\begin{tikzpicture}[scale=0.5,grow cyclic,shape=circle,minimum size = 2pt,inner sep=1.5pt,level distance=13mm,
                    cap=round]
\node [white] {} child [\A] foreach \A in {black,black,black}
    { node [white] {} child [color=\A!50!\B] foreach \B in {black,black}
        { node  {} child [color=\A!50!\B!50!\C] foreach \C in {black,black}
            { node {} }
        }
    };

\draw  (0,0) circle (2.5pt);        
\filldraw  (-0.65,1.11) circle (2.5pt);        
\filldraw  (-0.65,-1.11) circle (2.5pt);        
\filldraw [white] (1.3,0) circle (2pt);
\draw  (1.3,0) circle (2.5pt);                
\filldraw [white] (2.135,1) circle (2pt);    
\filldraw [gray] (2,1) circle (3pt);    
\filldraw [gray] (2.2,1.15) circle (3pt);    
\filldraw [gray] (2.23,0.87) circle (3pt);    
\draw [dashed,->] (1.8,1.1) to [bend left=40] (2,2.2);    
\draw [dashed,->] (2.4,0.7) to [bend right=40] (3.4,1.1);    

\end{tikzpicture}

}


\end{center}
\caption{Realization of the frog model on $\mathbb{T}_2$. Active and inactive particles are represented in gray and black, resp.}
\label{fig:frogs}
\end{figure}
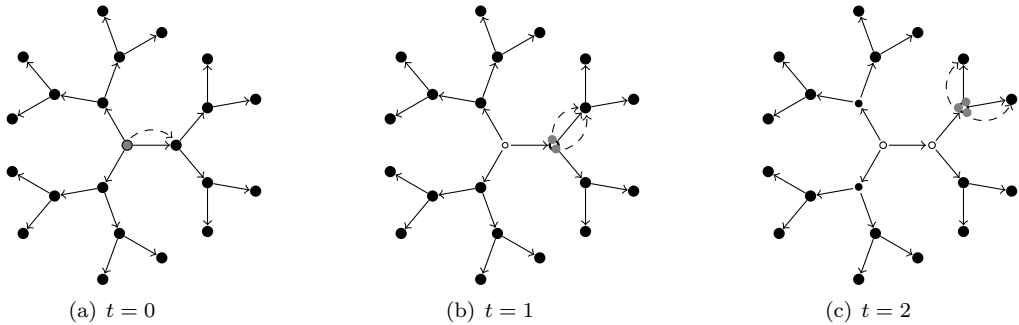


\smallskip
\begin{obs}
{Several papers consider models in which the process is started with a random number of particles at each vertex of the tree. We can do this as well, but in order to simplify the presentation, we only presented the results in the one-per-site case. The reader interested in an extension to random initial configuration may consult \cite[Section 5]{lebensztayn/machado/popov/2005} where the changes to be done in the proofs are explained. }
\end{obs}

\subsection{Results concerning frogs on directed trees}

Let $\theta(d,c,q)$ denote the probability of survival of this model. For this model also, a coupling argument shows that, for $d\ge2$,  $\theta(d,c,q)$ is monotone nondecreasing in $q$ and thus there exists a critical parameter
\begin{equation}\label{eq:qc}
q_c(d,c):=\inf\{q>0:\theta(d,c,q)>0\}.
\end{equation} 

Our first main result is an equality that the critical parameter $q_c$ must satisfy.

\smallskip
\begin{thm}\label{theo}
For any fixed $d\ge2$ and $c\in(0,1]$, the critical value $q_c=q_c(c,d)$ is the solution of the equality
\begin{equation}\label{eq:localization}
\sum_{k\ge1}c(dq)^k\prod_{i=1}^{k-1}(1-cq^i)=1.
\end{equation}
In particular, $q_c\in(0,1/d)$.
\end{thm}

\smallskip
\begin{obs}
{We point out that there are different ways to present condition \eqref{eq:localization}.} The q-Pochhammer symbol is defined as $(a;x)_{k}:=\prod_{i=0}^{k-1}(1-a x^i)$, with the convention that $(a;x)_{0}=1$. Then, Theorem \ref{theo} states that $q_c$ is the solution of $\sum_{k\ge1}c(dq)^k (cq;q)_k=1.$ An alternate way of stating \eqref{eq:localization} is considering $f_k^{(q)}:=cq^k\prod_{i=1}^{k-1}(1-cq^i),k\ge1$, a probability distribution indexed by $q$ (recall that $c$ and $d$ are fixed). Thus, letting $N_q$ be a random variable with distribution $f_k^{(q)},k\ge1$, $q_c$ is such that $\mathbb{E}d^{N_{q_c}}=1$.
\end{obs}

\smallskip
We obtain the following bounds relating $q_c$, $d$ and $c$.

\smallskip
 \begin{coro}\label{cor:precise}
\begin{equation}\label{eq:precise}
\frac{(c +1)\left(1- \sqrt{1-4q_c\frac{c^2}{(c+1)^2} }\right)}{2q_c^2 c^2}\le d\le \frac{(c+1)(1 - \sqrt{1-4\frac{c^2}{(c+1)^2}q_c(q_c+1)}}{2 c^2q_c^2(q_c + 1)}.
\end{equation}
\end{coro}

\smallskip
In order to identify bounds for $q_c$ as a function of $c$ and $d$, we have to revert the above bounds in order to isolate $q_c$. We did this using the software \emph{Mathematica}. Although the obtained bounds are explicit functions of c and d, the expressions are not very friendly, nor really informative, so for the sake of simplicity we do not  include them here. The interested readers may find them in \cite[see Display (5) therein]{gallo/rodriguez/2017}, which is a previous version of this paper.
We used these expressions to obtain our tightest numerical bounds, which are presented in the columns ``Corollary \ref{cor:precise}'' of the tables in the sequel. For instance,
Table \ref{fig:cone_percolation} gives numerical bounds for our frog model in the case where $c=1$. 
 
The next corollary presents more explicit expressions.



\smallskip
\begin{coro}\label{theo2} We have, {for any $d\geq 3$,}
\[ 
\frac{1}{d(c+1)-(c/(c+1))^2}\le q_c (d)\le\frac{F(c,d)-\sqrt{F^2(c,d)-224c^2(c+1)^2}}{16c^2}
\]
where $F(c,d):=7d(c+1)^3-8 c^2$. The lower bound also holds for $d=2$.
\end{coro}

\section{Applications}\label{sec:applications}

Now we discuss the applications of our results to some models of the literature. {While the applications that we will present in Subsections \ref{sec:original} and \ref{sec:removal} are direct consequences of the above results, the case of the  \emph{cone percolation}, that we first investigate, is an application of the method of proof. }

\subsection{Cone percolation  with geometric radius}

The first application we discuss is an improvement of both, upper and lower bounds, for the critical probability of a long range percolation model on $\overrightarrow{\mathbb{T}}_d$ called  cone percolation model, see  \cite{junior/machado/zuluaga/2014}. There is a random variable $X_v$ associated to each vertex $v\in\mathcal{V}:=\mathcal{V}(\mathbb{T}_d)$, representing a radius of propagation of a piece of information. The $X_v$'s are independent copies of $X\sim\textrm{Geo}(1-p)$, for some $p\in(0,1)$, and it is assumed that the information propagates through the graph as follows. At time $0$ only the root has the information. At time one all the vertices at distance of at most $X_{{\bf 0}}$ of the root of the tree are informed. At each step, each newly informed vertex $v$ will inform all non-informed vertices $v'>v$ such that $d(v,v')\le X_v$. \cite{junior/machado/zuluaga/2014} proved that there exists a critical value $p_c^{\text{\bf cp}}$ (the superscript ``cp'' stands for cone percolation) above which infinitely many vertices are informed (i.e., the model percolates) with positive probability. 

{Our interest is to compare this information propagation with the frog model on the directed tree, making $c=1$ and $q=p$ in \eqref{eq:Rn}.  A quick look to the first step of the proof of Theorem \ref{theo} shows that \eqref{eq:suf1} and \eqref{eq:suf2} are also valid using the dynamic of the cone percolation in place of our frog model. The remaining of the proof only relies on the dynamic along one single branch, which is the same in both models with our choice of the parameters. Thus, making  $c=1$ and $q=p$, the results of our frog model are valid for the geometric cone percolation, and we obtain the following result. }

\begin{prop}\label{coro:cone}
The critical parameter $p^{\text{\bf cp}}_c$ is the solution in $q$ of \eqref{eq:localization} with $c=1$. More explicitly,
we have $0.266667\le p_c^{\text{\bf cp}}(2)\le 0.277206$ and for $d\ge3$
\begin{align*}
\frac{1}{2d-\frac{1}{4}}\le p_c^{\text{\bf cp}}(d)\le \frac{(7d-1)(1-\sqrt{1-\frac{14}{(7d-1)^2}})}{2}.
\end{align*}
\end{prop}
These  bounds improve  the ones of \cite{junior/machado/zuluaga/2014} who obtained
\begin{equation}\label{eq:cone}
1/(2d)\le p_c^{\text{\bf cp}}(d)\le 1-\sqrt{1-1/d}.
\end{equation}
{We refer to Table \ref{fig:cone_percolation} for a  comparison between \eqref{eq:cone} and Proposition \ref{coro:cone}.}


\begin{table}[h]
{\footnotesize
\begin{center}
\begin{equation*}
\begin{array}{c|ccc|ccc|}
&  \multicolumn{3}{|c|}{\smash{\raisebox{.2\normalbaselineskip}{\text{Lower Bounds}}}} &  \multicolumn{3}{|c|}{\smash{\raisebox{.2\normalbaselineskip}{\text{Upper Bounds}}}}\\
d & \text{Corollary \ref{cor:precise}} & \text{Proposition \ref{coro:cone}} & \text{Known (2014)} & \text{Corollary \ref{cor:precise}} & \text{Proposition \ref{coro:cone}} & \text{Known (2014)}\\\hline
2 & 0.269594 & 0.266667 & 0.250000 & 0.277206 & 0.277206 & 0.292893\\
3 & 0.174659 & 0.173913 & 0.166667 & 0.176343 & 0.176559 & 0.183503 \\
4 & 0.129326 & 0.129032 & 0.125000 & 0.129961 & 0.130258& 0.133975 \\
5 & 0.102709 & 0.102564 & 0.100000 & 0.103015 & 0.103255& 0.105573 \\
6 & 0.085188 & 0.085106 & 0.083333 & 0.085358 & 0.085544 & 0.087129 \\
7 & 0.072777 & 0.072727&0.071428 & 0.072882 & 0.073027 & 0.074179\\
8 & 0.063525 & 0.063492 & 0.062500 & 0.063594 & 0.063710 &0.064585\\
9 & 0.056361 & 0.056338 & 0.055556 & 0.056408 & 0.056503 & 0.057191\\
10 & 0.050649 & 0.050632 & 0.050000 & 0.050684 & 0.050762 &0.051316\\
15 & 0.033618 & 0.033613 & 0.033333 & 0.033628 & 0.033664 & 0.033908\\
20 & 0.025159 & 0.025157 &0.025000 & 0.025163 & 0.025184 & 0.025320\\
30 & 0.016737 & 0.016736 & 0.016667& 0.016738 & 0.016748 & 0.016807\\
50 & 0.010025 & 0.010025 & 0.010000 & 0.010025 & 0.010028 & 0.010050\\
100 & 0.005006 & 0.005006 & 0.005000 & 0.005006& 0.005007 & 0.005012\\\hline\end{array}
\end{equation*}
\end{center}}
\caption{{The first and the forth columns present numerical values obtained reverting Corollary \ref{cor:precise} with \emph{Mathematica} when $c=1$. The columns ``Known (2014)'' use Display \eqref{eq:cone}.}
\label{fig:cone_percolation}}
\end{table}

\subsection{Improved upper bounds for the {original}  frog model and its self-avoiding version}\label{sec:original}

Here we present an improvement of the known upper bounds for the critical parameter of the original frog model, as well as the self-avoiding frog model, on $\mathbb{T}_d$. Let us start with the frog model with one-per-site  configuration, i.i.d. geometric lifetimes of parameter $1-p$, for some $p\in (0,1)$, and denote by $p^{\text{\bf o}}_c(d)$ its critical parameter, for $d\geq 2$ (the superscript ``o'' stands for ``original'', to point out that it refers to the original model). A useful result to obtain an upper bound for $p^{\text{\bf o}}_c(d)$ is Lemma 2.1 of \cite{lebensztayn/machado/popov/2005}. It states that, for any two vertices $u$ and $v$  {such that $u<v$ and $d(u,v)=n\ge1$, vertex $v$} will be visited by the active particle starting at $u$ with probability $r^n$, where

\[
r=r(p):=\frac{d+1-\sqrt{(d+1)^2-4dp^2}}{2dp}.
\]
Our frog model on $\overrightarrow{\mathbb{T}}_d$ with $c=1$ and $q=r(p)$ can be coupled to the ``original frog model'' in such a way that our model is \emph{below}. To do this, we start with the  one-per-site configuration for both models (our on $\overrightarrow{\mathbb{T}}_d$ and the ``original'' on $\mathbb{T}_d$) and we realize both processes in such a way that an active particle hits a given vertex in $\overrightarrow{\mathbb{T}}_d$ only if the corresponding particle on $\mathbb{T}_d$ also hits this vertex. 
Thus, if our frog model on $\overrightarrow{\mathbb{T}}_d$ with $q=r(p)$ and $c=1$ survives,  the ``original'' frog model with $p$ survives as well. So we can use the upper bound of Corollary \ref{theo2} {(and revert Corollary \ref{cor:precise} for $d=2$}) to get the following result.

\begin{prop}\label{coro:original} We have $p_c^{\text{\bf o}}(2)\le 0.7208{36}$ and for $d\ge3$
\begin{align*}
&p_c^{\text{\bf o}}(d)\le \frac{(d+1)[(7d-1)-\sqrt{(7d-1)^2-14}] }{d(7d-1)^2-7d+2-d(7d-1)\sqrt{(7d-1)^2-14}}.
\end{align*}
\end{prop}

\smallskip
This upper bound improves the one of \cite{lebensztayn/machado/popov/2005} who obtained 
\begin{equation}\label{lit_o}
p_c^{\text{\bf o}}(d)\le (d+1)/(2d).
\end{equation}
The left halfpart of Table \ref{fig:original_frog} gives numerical comparison between the bounds. We point out that their bound was an improvement over that of \cite{fontes/machado/sarkar/2004} who obtained $p_c^{\text{\bf o}}(d)\le(d+1)/(2d-2)$.

 
 \smallskip
Now we consider the self-avoiding version of the frog model on $\mathbb{T}_d$ which was introduced in \cite{lebensztayn/machado/martinez/2006}. The only difference with respect to the preceding model is that each particle performs a  self-avoiding random walk on $\mathbb{T}_d$ when it is activated. 
Again, a coupling argument allows us to compare this model with the frog model on $\overrightarrow{\mathbb{T}}_d$, but taking now $c=d/(d+1)$ and $q=p/d$ in \eqref{eq:Rn}. In the coupled versions,  we remove the activated particles of $\overrightarrow{\mathbb{T}}_d$ for which the corresponding particle on $\mathbb{T}_d$ jumps in the direction of the root. In any other case both particles do the same trajectory away from the root and with the same geometric lifetimes. 
We can see that survival for our model implies survival for the self-avoiding frog model. So here also, {denoting by $p_c^{\text{\bf sa}}(d)$ (the superscript ``sa'' stands for ``self-avoiding'') the critical parameter of the self-avoiding model,} 
we can use the upper bound of Corollary \ref{theo2} {(and revert Corollary \ref{cor:precise} for $d=2$}) to get the following result.

\begin{prop} \label{coro:selfrep}
We have $p_c^{\text{\bf sa}}(2)\le0.648046$ and for $d\ge3$
\begin{align*}
&p_c^{\text{\bf sa}}(d)\le(d+1)^2\frac{F(d/(d+1),d)-\sqrt{F^2(d/(d+1),d)-224c^2(d/(d+1)+1)^2}}{16d}.
\end{align*}
where we recall that where $F(c,d):=7d(c+1)^3-8 c^2$.
\end{prop}

 \smallskip
This upper bound improves the one of \cite{lebensztayn/machado/martinez/2006}, who obtained
\begin{equation}\label{lit_s_a}
p_c^{\text{\bf sa}}(d)\le {\frac{2d+1-\sqrt{4d^2-3}}{2}}. 
\end{equation}
We refer to Table \ref{fig:original_frog} for numerical comparisons between \eqref{lit_s_a} and Proposition \ref{coro:selfrep}.
  \begin{table}[h]
{\footnotesize
\begin{center}
\begin{equation*}
\begin{array}{c|ccc|ccc|}
&  \multicolumn{3}{|c|}{\smash{\raisebox{.2\normalbaselineskip}{\text{Upper bound ``original frog'' model}}}} &  \multicolumn{3}{|c|}{\smash{\raisebox{.2\normalbaselineskip}{\text{Upper bound ``Self-avoiding frog'' model}}}}\\
d & \text{Corollary \ref{cor:precise}} & \text{Proposition \ref{coro:original}} & \text{Known (2005)} & \text{Corollary \ref{cor:precise}} & \text{Proposition \ref{coro:selfrep}} & \text{Known (2006)}\\\hline
2 & 0.720836 & 0.7208{36} & 0.750000 & 0.648045 & 0.648045 & 0.697224\\
3 &  0.645182 & 0.645837 & 0.666667 & 0.599063 & 0.600229 & 0.627719\\
4 &   0.608681 & 0.609897 & 0.625000  & 0.574870 & 0.576225 & 0.594875\\
5 & 0.586944 &  0.588174&0.600000 &  0.560271 & 0.561544 & 0.575571\\
6 & 0.572482 & 0.573624& 0.583333 & 0.550468 & 0.551621 & 0.562829\\ 
7 & 0.562156 & 0.563197&0.571429 & 0.543421 & 0.544461 & 0.553778\\  
8 & 0.55441 &0.555358 & 0.562500 & 0.538107 & 0.539048 & 0.547013\\ 
9 & 0.548384 & 0.549249 & 0.555556 &0.533955 & 0.534812 & 0.541764\\ 
10 & 0.543561 & 0.544355 & 0.550000 &  0.530620 & 0.531406 & 0.537571\\ 
15 & 0.529076 &0.529632 &0.533333  & 0.520543 & 0.521093 & 0.525021\\ 
20 &0.521822  & 0.522248 & 0.525000 &0.515458 & 0.515881 & 0.518759\\ 
30 & 0.514559 & 0.514848 &  0.516667 & 0.510341 & 0.510628 & 0.512503\\ 
50 & 0.508741 & 0.508917 & 0.510000 &0.506222 & 0.506397 & 0.507501\\ 
100 & 0.504373 & 0.504461 & 0.505000 &0.503118 & 0.503206 & 0.503750\\\hline
\end{array}
\end{equation*}
\end{center}
\caption{Numerical values of upper bounds for the {original} frog model and its self-avoiding version. The third and the sixth columns are the bounds obtained respectively by (Lebensztayn et al., 2005) and (Lebensztayn et al., 2006).}
\label{fig:original_frog}}
\end{table} 

\subsection{Frog model with removal at visited vertices}\label{sec:removal}

In the previous subsection we have presented some improvements for the original frog model and its self-avoiding version on $\mathbb{T}_d$. As a final application of our results we discuss another version of the frog model which has not been explored on infinite graphs. It is a frog model with $i.i.d.$ geometric lifetimes of parameter $1-p$, in which we remove any particle which did not activate anybody for $L$ times, where $L\geq 1$. This modification was suggested by \cite{popov/2003} and as far as we known the only rigorous results on infinite graphs are proved in $\mathbb{Z}$ for some related models, see for instance \cite{Lebensztayn/Machado/Martinez/2016} and references therein. On the other hand, the model has been well studied on some finite graphs. In such case, the issue of interest is the study of the final proportion of visited vertices at the end of the process. In this direction, \cite{Alves2006} obtained the first results for the model defined on a complete graph with $L=1$ and $p=1$ by mean of a mean field approximation analysis and computational simulations. They work was generalized later by \cite{Kurtz/Lebensztayn/Leichsenring/Machado/2008} for $L\geq 1$ and $p=1$ in the form of limit theorems obtained for the proportion of visited vertices at the absorption time of the process as the size of the graph goes to $\infty$. More recently, \cite{Lebensztayn/Rodriguez/2013} stated the connection between this model and the well known Maki-Thompson rumor model (see \cite{lebensztayn/machado/rodriguez/2011b}).  {In view of the connection obtained by \cite{Lebensztayn/Rodriguez/2013}, one may consider the model presented here as a rumor process on a moving population.}    

{If we consider this model on $\mathbb{T}_d$, starting from a one-per-site  configuration then, a realization of the resulting process {when $L=1$} coincides with a realization of our general frog model on $\overrightarrow{\mathbb{T}}_d$, making $c=1$ and $q=p/(d+1)$ in \eqref{eq:Rn}.
Therefore, denoting by $p^{\text{\bf r}}_c(d)$ the  critical parameter of this model  (the superscript ``r'' stands for ``removal'') with $L=1$, we directly obtain the following proposition. }

\begin{prop} The critical parameter $p^{\text{\bf r}}_c$ is the solution in $p$ of \eqref{eq:localization} with $c=1$ and $q=p/(d+1)$. More explicitly,
we have $0.8\le p_c^{\text{\bf r}}(2)\le0.831619$, and for $d\ge3$
\begin{align*}
\frac{d+1}{2d-1/4}\le p_c^{\text{\bf r}}(d)\le \frac{(d+1)(7d-1)(1-\sqrt{1-\frac{14}{(7d-1)^2}})}{2}.
\end{align*}

\end{prop}

\section{Interlude: Renewal convergence rates}\label{sec:renewal}

Our proofs will use a parallel between information propagation on $\bN$ and undelayed renewal sequences. For this reason we dedicate this section to the description of some aspects of renewal theory. Let ${\bf T} = (T_n)_{n\ge1}$ be an i.i.d. sequence of  r.v's, taking values in $\{1,2,\ldots\}\cup\{\infty\}$ with common distribution $(f_k)_{k\in\{1,2,\ldots\}\cup\{\infty\}}$. The undelayed renewal sequence is the $\{0,1\}$-valued stochastic chain ${\bf Y}=(Y_n)_{n\ge0}$ defined through $Y_0=1$ and, for any $n\ge1$, $Y_n={\bf1}\{T_1+\ldots+T_i=n\,\,\textrm{for some}\,\,i\}$. The distribution $(f_k)_{k\in\{1,2,\ldots\}\cup\{\infty\}}$ is called inter-arrival distribution. 
 The well-known renewal theorem states that 
\[
\mathbb{P}(Y_n=1)\rightarrow \frac{1}{\mathbb{E}T},
\]
with the convention that $1/\infty=0$.  The question of  identifying the rate at which this convergence holds (renewal convergence rate), based on the inter-arrival distribution is a very classical one (see the introduction of \cite{giacomin/2008} for a rapid survey). 

For our purposes, let us first observe that  the renewal property implies that
\begin{align*}
\mathbb{P}(Y_{n}=1)\mathbb{P}(Y_m=1)&=\mathbb{P}(Y_{n+m}=1|Y_{n=1})\mathbb{P}(Y_n=1)\\&=\mathbb{P}(Y_n=1,Y_{n+m}=1)\\&\le \mathbb{P}(Y_{n+m}=1).
\end{align*}
This means in particular that $\log \mathbb{P}(Y_n=1)$ is super-additive, thus by Fekete's lemma, $\lim_n\frac{1}{n}\log \mathbb{P}(Y_n=1)$ exists. We introduce the renewal convergence rate of the process, $\gamma$, defined through
\[
\log\gamma:=-
\lim_n\frac{1}{n}\log \mathbb{P}(Y_n=1).
\] 
Naturally, the value of  $\log\gamma$ depends on the inter-arrival distribution. Here we focus (having in mind a future application to the frog model described in Section \ref{ss:frogs_trees}) on inter-arrival distribution having the property that the hazard rate
\[
h_k:=\frac{f_k}{\sum_{i\ge k}f_i}=c q^k ,\,\,k\ge1
\]
for some $c>0$ and $q\in(0,1)$. Reversely, we have the following expression for $f_k$, 
\begin{equation}\label{eq:distribution}
f_k=c q^k\prod_{i=1}^{k-1}(1-c q^i)\,,\,\,\,\,k\ge1
\end{equation}
with the convention that $\prod_{i=1}^{0}(1-c q^i)=1$. This is a defective probability distribution, since 
\begin{equation}\label{eq:distribution2}
\mathbb{P}(T\ge n)=\sum_{k\ge n}f_k=\prod_{i=1}^{n-1}(1-c q^i)\rightarrow \prod_{i\ge1}(1-c q^i)>0.
\end{equation}
So we have that $f_\infty:=\mathbb{P}(T=\infty)=\prod_{i\ge1}(1-c q^i)>0$. 

\smallskip
The proofs of our results will make use of the two following lemmas. 

\begin{lem}\label{prop:renewal} For a renewal process with inter-arrival distribution given by \eqref{eq:distribution} we have
 \[
 \sum_{k\ge1}\gamma^{k}cq^k\prod_{i=1}^{k-1}(1-c q^i)=1.
 \]
\end{lem}

\begin{proof}
Theorem 3.5 of \cite{bremaud/1999} states that, if for some defective distribution $(f_k)_{k\in\{1,2,\ldots\}\cup\{\infty\}}$ there exists some $\alpha>1$ such that
\begin{equation}\label{eq:condition_defective}
F(\alpha):=\sum_{n\ge1}\alpha^nf_n=1
\end{equation}
then the limit $\lim_n\alpha^n\mathbb{P}(Y_n=1)=(\sum_nn\alpha^nf_n)^{-1}$. 

In our case (recall that $f_n,n\ge1$ satisfies \eqref{eq:distribution}) we can prove that \eqref{eq:condition_defective} actually holds for some $\alpha$. This fact is proved for instance in \cite[Proof of Proposition 2(iii)]{bressaud/fernandez/galves/1999a}. For completeness we include the argument here. Observe first that $F(1)=\mathbb{P}(T<\infty)<1$. Moreover, by \eqref{eq:distribution} and \eqref{eq:distribution2}, we have that $f_n/(c q^n)\rightarrow \prod_{i\ge1}(1-c q^i)>0$, meaning in particular that the radius of convergence of $F$, $\lim_nf_n^{-1/n}=\lim_n(cq^n)^{-1/n}=q^{-1}$, which is larger than $1$.   Thus $F(1)<1$ and we can find $\delta\in(1,q^{-1})$ such that $1<F(\delta)<+\infty$. By continuity of $F$, there exists $\alpha$ such that $F(\alpha)=1$. 

So \eqref{eq:condition_defective} is true for some $1<\alpha<q^{-1}$, meaning moreover that  $(\sum_nn\alpha^nf_n)^{-1}\in(0,\infty)$. Using Theorem 3.5 of \cite{bremaud/1999}, we have
\[
0=-\lim_n \frac{\log \alpha^n\mathbb{P}(Y_n=1)}{n}=-\lim_n \frac{\log \mathbb{P}(Y_n=1)}{n}-\lim_n \frac{\log \alpha^n}{n}=-\lim_n \frac{\log \mathbb{P}(Y_n=1)}{n}-\log \alpha,
\]
and therefore $-\lim_n \frac{\log \mathbb{P}(Y_n=1)}{n}=\log\alpha>0$, and $\alpha$ is indeed the renewal convergence rate $\gamma$ of the process. This concludes the proof of the proposition.
\end{proof}

For the next lemma, consider $c$ is fixed, and thus $\gamma$ can be seen as a function of $q$.

\begin{lem}\label{lem1}
$\gamma$ is a {continuous function of $q$ on $q\in(0,1/(cd))$.}
\end{lem}

\begin{proof} {Fix $q\in(0,1/(cd))$} and consider any sequence $q_\epsilon\rightarrow q$ as $\epsilon\rightarrow0$, where  {$q_\epsilon\in(0,1/(cd))$}.  Define naturally $f_{k,\epsilon}:=c q_\epsilon^k\prod_{i=1}^{k-1}(1-c q_\epsilon^i), k\ge1$ and observe that $f_{k,\epsilon}\rightarrow f_k$ for any $k\ge1$. For any $\epsilon$, we can use the proof of Lemma \ref{prop:renewal} to show that there exists a unique solution  $\alpha>1$ for the equation
\[
F_{\epsilon}(\alpha):=\sum_{n\ge1}\alpha^nf_{n,\epsilon}=1. 
\]
We naturally denote this solution by $\gamma(q_\epsilon)$. Observe moreover that we can find $\delta\in(1,q^{-1})$ such that $F(\delta)>1$ and $F_\epsilon(\delta)<\infty$ for any sufficiently small $\epsilon$. In these conditions, it is proved in \cite{petersson/silvestrov/2013} that  $\gamma(q_\epsilon)\rightarrow\gamma(q)$, showing that $\gamma$ is a continuous function of $q$ on {$(0,1/(cd))$.}

\end{proof}

\section{Proofs of the main results}\label{sec:proofs}

\begin{proof}[{\bf Proof of Theorem \ref{theo}}]

Let us draw the main steps of this proof. (1) We transform the problem of  {survival of the frog model} on the tree into that of controlling the propagation along one single branch (see displays \eqref{eq:suf1} and \eqref{eq:suf2} below). (2) We use a result of \cite{gallo/garcia/junior/rodriguez/2014} which implies that the probability that the dynamic along one single branch reaches distance $n$ is equal to the probability that a specific renewal process renews at time $n$. This allows us to relate to the preceding section, and specifically to prove that $\gamma_c:=\gamma(q_c)=d$. 
(3) We conclude the proof using Lemma \ref{prop:renewal} of the preceding section. \\
%
%

For the first step of the proof, we shall use a simple union bound for the lower bound on the critical parameter, and a classical coupling with branching processes for the upper bound. 
Let us introduce $A_n$ and $A_\infty$ denoting respectively $\{\textrm{a frog at distance $n$ of the root is activated}\}$ and $\{\textrm{infinitely many frogs are activated}\}$. Fix $d\ge2$ and $c>0$ and consider our frog model parametrized by $q$.  Naturally, we have $\theta(q)=\mathbb{P}_q(A_{\infty})=\lim_n\mathbb{P}_q(A_n)$. For any $v\in\mathcal{V}$, let $A^v:=\{\textrm{the frog of vertex $v$ is activated}\}$.

To find a lower bound for $q_c$, just observe that
$
\mathbb{P}_q(A_n)=\mathbb{P}_q(\cup_{v:d(0,v)=n}A^v)\le d^np_{q,n}
$
where $p_{q,n}$ denotes the common value (by symmetry) of  the $\mathbb{P}_q(A^v)$'s for any $v$ at distance $n$ of the root.  Thus, 
\begin{equation}\label{eq:suf1}
d^np_{q,n}\rightarrow0\Rightarrow q<q_c.
\end{equation}
 {In other words,  to find a non-trivial lower bound for $q_c$, it is sufficient to find a value $q>0$ such that the left side of \eqref{eq:suf1} holds.} 

In order to find an upper bound for $q_c$, we will couple our frog model, rescaled by some length $n\ge1$, with a branching process. {The coupling is the same as the one used in \cite[Section 3]{lebensztayn/machado/popov/2005} and \cite[Section 3.1]{lebensztayn/machado/martinez/2006}, so we only describe it informally here}. The individuals of the branching process are identified with vertices of the tree recursively as follows. Originally, there is one individual, which is the root. Its offspring is identified with the set of vertices  {at distance $n$ from the root which are visited by active frogs at some time of the process. That is, a vertex $v$ is identified with an individual of the offspring of the root provided $d(0,v)=n$ and the event $A^v$ occurs.} 
Recursively, we start one active frog from each vertex $v$ of the  offspring of the $(k-1)^{\text{th}}$ generation, and identify its offspring (of the $k^{\text{th}}$ generation) with the set of vertices {visited} by active frogs  and which are at distance $n$ from  $v$.

%
%
The resulting branching process is below our frog model in the sense that  if it survives, our model survives also. Notice that the offspring distribution of the first generation (children of the root) is different from the others, since on $\overrightarrow{\mathbb{T}}_d$, there is $(d+1)d^{n-1}$ vertices at distance $n$ from the root (whereas there are $d^n$ at distance $n$ from any other vertex).  This has no effect on the fact that the branching process survives if there exists $N\ge1$ such that $\mathbb{E}_p (\sum_{v:d(0,v)=N}{\bf1}\{A^v\})=d^Np_{q,N}>1$. This is also a sufficient condition for our frog model, so
\begin{equation}\label{eq:suf2}
\exists N:d^Np_{q,N}>1\Rightarrow q{\ge}q_c.
\end{equation}

With \eqref{eq:suf1} and $\eqref{eq:suf2}$ in hand, we come to the second step of the proof, which is to get informations about $p_{q,n}$, $n\ge1$, the probability that a given vertex at distance $n$ from the root is activated. In other words, we want to investigate the way the process propagates along one single branch of the tree, since this is the only way to reach this vertex. At this stage we make a comparison with another process of the literature, originally introduced by \cite{junior/machado/zuluaga/2011} under the name of ``firework process'' to model information spreading on $\bN$. {So let us briefly describe this model. We  start with one spreader  at site $0$ and ignorants at all the other sites of $\bN$. The spreaders transmit the information within a random distance, which are independent copies of an $\bN$-valued random variable $D$, on their right, as illustrated in Figure \ref{FIG:FP}.}

{Lemma 1 in \cite{gallo/garcia/junior/rodriguez/2014} states that the probability that the firework process on $\bN$ reaches site $n$ is equal to the probability that an undelayed renewal sequence (see Section \ref{sec:renewal} above) ${\bf Y}$ with hazard rate 
\[
h_k:=\frac{f_k}{\sum_{i\ge k}f_i}=\mathbb{P}(D\ge k)
\]
renews at time $n$. }

\begin{figure}[h]
\label{FIG:FP}
\begin{center}
\begin{tikzpicture}[scale=0.8]

\draw (-3,-2) -- (8.5,-2);
\draw (-3,-2.3) ;
\draw (8.8,-2) node {...};

\draw (-3,-2.3) node[below,font=\footnotesize] {$0$};

\filldraw [black] (-3,-2) circle (3pt) ;
\draw [very thick] (-2,-2) circle (3pt);
\filldraw [white] (-2,-2) circle (3pt);
\draw [very thick] (-1,-2) circle (3pt);
\filldraw [white] (-1,-2) circle (3pt);
\draw [very thick] (0,-2) circle (3pt);
\filldraw [white] (0,-2) circle (3pt);
\draw [very thick] (1,-2) circle (3pt);
\filldraw [white] (1,-2) circle (3pt);
\draw [very thick] (2,-2) circle (3pt);
\filldraw [white] (2,-2) circle (3pt);
\draw [very thick] (3,-2) circle (3pt);
\filldraw [white] (3,-2) circle (3pt);
\draw [very thick] (4,-2) circle (3pt);
\filldraw [white] (4,-2) circle (3pt);
\draw [very thick] (5,-2) circle (3pt);
\filldraw [white] (5,-2) circle (3pt);
\draw [very thick] (6,-2) circle (3pt);
\filldraw [white] (6,-2) circle (3pt);
\draw [very thick] (7,-2) circle (3pt);
\filldraw [white] (7,-2) circle (3pt);
\draw [very thick] (8,-2) circle (3pt);
\filldraw [white] (8,-2) circle (3pt);


\draw [->] (-3,-1.8) to [bend left=30] (0,-1.8);


\draw (-3,-4) -- (8.5,-4);


\draw (-3,-4.3) node[below,font=\footnotesize] {$0$};

\draw (0,-4.6) ;
\draw (8.8,-4) node {...};

\filldraw [gray] (-3,-4) circle (3pt);
\filldraw [black] (-2,-4) circle (3pt);
\filldraw [black] (-1,-4) circle (3pt);
\filldraw [black] (0,-4) circle (3pt);
\draw [very thick] (1,-4) circle (3pt);
\filldraw [white] (1,-4) circle (3pt);
\draw [very thick] (2,-4) circle (3pt);
\filldraw [white] (2,-4) circle (3pt);
\draw [very thick] (3,-4) circle (3pt);
\filldraw [white] (3,-4) circle (3pt);
\draw [very thick] (4,-4) circle (3pt);
\filldraw [white] (4,-4) circle (3pt);
\draw [very thick] (5,-4) circle (3pt);
\filldraw [white] (5,-4) circle (3pt);
\draw [very thick] (6,-4) circle (3pt);
\filldraw [white] (6,-4) circle (3pt);
\draw [very thick] (7,-4) circle (3pt);
\filldraw [white] (7,-4) circle (3pt);
\draw [very thick] (8,-4) circle (3pt);
\filldraw [white] (8,-4) circle (3pt);

\draw [->] (-2,-3.8) to [bend left=30] (1,-3.8);
\draw [->] (0,-3.8) to [bend left=30] (2,-3.8);


\draw (-3,-6) -- (8.5,-6);


\draw (-3,-6.3) node[below,font=\footnotesize] {$0$};

\draw (4,-6.6) ;
\draw (8.8,-6) node {...};

\filldraw [gray] (-3,-6) circle (3pt);
\filldraw [gray] (-2,-6) circle (3pt);
\filldraw [gray] (-1,-6) circle (3pt);
\filldraw [gray] (0,-6) circle (3pt);
\filldraw [black] (1,-6) circle (3pt);
\filldraw [black] (2,-6) circle (3pt);
\draw [very thick] (3,-6) circle (3pt);
\filldraw [white] (3,-6) circle (3pt);
\draw [very thick] (4,-6) circle (3pt);
\filldraw [white] (4,-6) circle (3pt);
\draw [very thick] (5,-6) circle (3pt);
\filldraw [white] (5,-6) circle (3pt);
\draw [very thick] (6,-6) circle (3pt);
\filldraw [white] (6,-6) circle (3pt);
\draw [very thick] (7,-6) circle (3pt);
\filldraw [white] (7,-6) circle (3pt);
\draw [very thick] (8,-6) circle (3pt);
\filldraw [white] (8,-6) circle (3pt);
\draw [->] (1,-5.8) to [bend left=30] (5,-5.8);
\draw [->] (2,-5.8) to [bend left=30] (4,-5.8);

\end{tikzpicture}
\end{center}
\caption{First three steps of the firework process. The white, gray and black sites represent ignorants, spreaders from earlier stages and current spreaders, respectively.}
\end{figure}
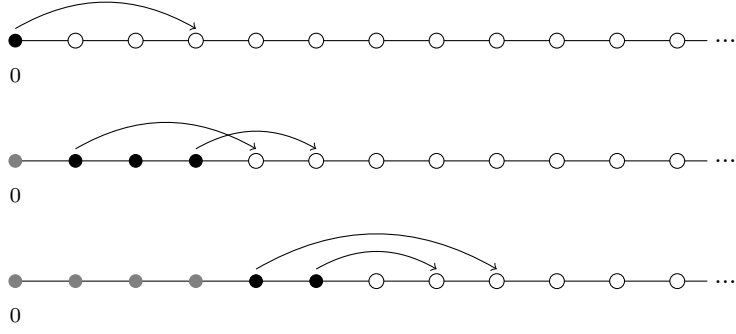

{Coming back to the dynamics of the frogs along one single branch, we identify the branch with $\bN$ and our active/inactive particles with spreaders/ignorants. The i.i.d. radius of transmission corresponds to the reach of the random walks of the activated frogs along the corresponding branch  in $\overrightarrow{\mathbb{T}}_d$. From this parallel,   we have $p_{q,n}=\mathbb{P}(Y_n=1)$ if we take  $\mathbb{P}(D\ge k)=cq^k$. }

As explained in Section \ref{fig:original_frog}, there exists $\gamma>0$ such that
$\log\gamma=-
\lim_n\frac{1}{n}\log \mathbb{P}(Y_n=1).
$
Suppose $\frac{d}{\gamma}>1$ (\emph{resp.} $\frac{d}{\gamma}<1$). There exists $\epsilon=\epsilon(d,\gamma) >0$ such that 
$\frac{d}{\gamma} e^{-\epsilon}>1$ (\emph{resp.} $\frac{d}{\gamma}e^{\epsilon}<1$). On the other hand, by the definition of limit, we know that for any $\epsilon>0$, there exist $N$ such that for any $n\ge N$ we have
$
e^{-n(\log\gamma+\epsilon)}\le \mathbb{P}(Y_n=1)\le e^{-n(\log\gamma-\epsilon)}
$
and thus, in particular
\[
\left(\frac{d}{\gamma}e^{-\epsilon}\right)^n\le d^n\mathbb{P}(Y_n=1)\le \left(\frac{d}{\gamma}e^{\epsilon}\right)^n.
\]
We therefore have the following sequence of implications:
\[
\frac{d}{\gamma}>1\Rightarrow\exists\epsilon:\frac{d}{\gamma} e^{-\epsilon}>1\Rightarrow\exists N:d^N \mathbb{P}(Y_N=1)>1.
\]
The other way around, 
\[
\frac{d}{\gamma}<1\Rightarrow\exists\epsilon:\frac{d}{\gamma} e^{ \epsilon}<1\Rightarrow \forall n\ge N,\,  d^n\mathbb{P}(Y_n=1)\le \left(\frac{d}{\gamma}e^{\epsilon}\right)^n{\rightarrow0}.
\]
{Thus, using \eqref{eq:suf2} and  \eqref{eq:suf1}, and recalling that $\gamma=\gamma(q)$ is a function of $q$, we proved that
\begin{align*}
\gamma(q)>d&\Rightarrow q<q_c\\
\gamma(q)<d&\Rightarrow q\ge q_c.
\end{align*}
Thanks to Lemma \ref{lem1}, we necessarily have $\gamma(q_c)=d$. Indeed, observe from the previous relations that $d\leq \lim_{q \nearrow q_c} \gamma(q) = \gamma(q_c)= \lim_{q \searrow q_c} \gamma(q) \leq d$.} To conclude the proof with the third and last step, using Lemma \ref{prop:renewal} which states that 
\[
 \sum_{k\ge1}\gamma^{k}(q)cq^k\prod_{i=1}^{k-1}(1-c q^i)=1.
 \]
 Thus we have
 \[
  \sum_{k\ge1}d^kcq_c^k\prod_{i=1}^{k-1}(1-c q_c^i)=1.
 \]

\end{proof}


\begin{proof}[{\bf Proof of Corollary \ref{cor:precise}}]

We have to find bounds for $  \sum_{k\ge1}d^kcq_c^k\prod_{i=1}^{k-1}(1-c q_c^i)=1$. 
For $n\ge2$ we can show that
\begin{equation}\label{eq:ref}
1-c q-c q^2\le\prod_{i=1}^{n-1}(1-c q^i)\le (1-c q),
\end{equation}
where the righthand side is trivial, the lefthand side can be shown easily by recursion. 
Now we have that
\begin{equation}\label{eq:1}
cdq_c+(1-c q_c-c q_c^2)c\sum_{n\ge2}(dq_c)^n\le 1\le cdq_c+(1-c q_c)c\sum_{n\ge2}(dq_c)^n.
\end{equation}
The first inequality of \eqref{eq:1} yields 
\begin{align*}
d
\le\frac{(c+1)(1 - \sqrt{1-4\frac{c^2}{(c+1)^2}q(q+1)}}{2c^2q^2(q + 1)}
\end{align*}
while the second inequality of \eqref{eq:1} yields 
\[
d\ge\frac{(c +1)\left(1- \sqrt{1-4q\frac{c^2}{(c+1)^2} }\right)}{2q^2c^2}.
\]


\end{proof}

\begin{proof}[{\bf Proof of Corollary \ref{theo2}}]
Using that $\sqrt{1-x}\le 1-\frac{x}{2}-\frac{x^2}{8}$ for $x\in[0,1]$, we get
\[
d\ge \frac{1+q\frac{c^2}{(c+1)^2} }{(c+1)q}.
\]
Reverting this inequality gives the lower bound of the corollary. 
On the other hand, using  that $\sqrt{1-x}\ge 1-\frac{x}{2}-\frac{x^2}{7}$ for $x\in[0,0.24]$, we obtain
\begin{equation}\label{eq:bound_better}
d\le\frac{1+\frac{8 c^2}{7(c+1)^2}q(q+1)}{(c+1)q}
\end{equation}
when $4c^2(c+1)^{-2}q (q +1)\le 0.24$. Recalling that $c\le1$, it is enough to prove that this inequality holds with $c=1$. From table \ref{fig:cone_percolation}, we see that the inequality $q_c (q_c +1)\le0.06$ holds for $d\ge3$. Reverting \eqref{eq:bound_better} gives the upper bound of the corollary. \\
 \end{proof}

\noindent{\bf Acknowledgments}
 We are particularly grateful to H. Lacoin for fruitful discussions concerning renewal theory which  lead to a substantial enhancement of the results.
 
 We thank two anonymous reviewers for the careful reading of the manuscript {(pointing a gap in an previous version of the paper) and for the comments and criticism concerning presentation. }
 
SG was supported by FAPESP (2015/09094-3), research fellowships CNPq (312315/2015-5) FAPESP (2017/07084-6) and Edital Universal CNPq (462064/2014-0). PMR was supported by FAPESP (2016/11648-0) and CNPq (304676/2016-0). 

\bibliographystyle{apt}
\bibliography{sandrobibli}

 \end{document}